\documentclass[10pt]{article}
\usepackage{amsthm}
\usepackage{amsmath,bm}
\usepackage{enumerate}
\usepackage{amssymb}
\usepackage{latexsym}
\usepackage{amsfonts}
\usepackage{color}
\usepackage{secdot}
\usepackage{enumitem}
\usepackage{mathrsfs}
\usepackage{subfigure}
\usepackage{epsfig}
\usepackage{multirow}
\usepackage{natbib}
\usepackage{graphicx}
\usepackage{multirow, booktabs}
\usepackage[table]{xcolor}
\newtheorem{theorem}{Theorem}[section]
\newtheorem{counterexample}{Counterexample}[section]

\newtheorem{assumption}{Assumption}[section]
\newtheorem{lemma}{Lemma}[section]
\newtheorem{remark}{Remark}[section]

\begin{document}

\title{Ordering results for random maxima and minima from two dependent Kumaraswamy generalized distributed samples}
\author{{\large { Sangita {\bf Das}$^{a}$\thanks {Email address:
                sangitadas118@gmail.com}~ and Narayanaswamy Balakrishnan$^{b}$\thanks {Email address:  bala@mcmaster.ca}}} \\
    { \em \small {\it$^{a}$Theoretical Statistics and Mathematics Unit, Indian Statistical Institute, Bangalore-560059, India}}\\
    {\em \small {\it $^{b}$Department of Mathematics and Statistics, McMaster University, Hamilton, Ontario L8S 4K1,
				Canada}}\\\\
{ \em{\it To appear in Statistics.}}
}
\date{}
\maketitle
\begin{center}{\bf Abstract}
\end{center}
{ Let $\{X_{1},\ldots,X_{N_1}\}$ and  $\{Y_{1},\ldots,Y_{N_2}\}$ be two sequences of interdependent heterogeneous samples, where for $i=1,\ldots,N_{1},$ $X_{i}\sim \text{Kw-G}(x, \alpha_{i}, \gamma_{i};G)$ and for $i=1,\ldots,N_{2},$ $Y_{i}\sim \text{Kw-G}(x, \beta_{i}, \delta_{i};H),$ where $G$ and $H$ are baseline distributions in the Kumaraswamy generalized model and $N_1$ and $N_2$ are two positive integer-valued random variables, independently of $X_{i}'$s and $Y_{i}'$s, respectively.} In this article, we establish several stochastic orders such as usual stochastic, hazard rate, reversed hazard rate, dispersive and likelihood ratio orders  between the random maxima ($X_{{N_1}:{N_1}}$ and $Y_{{N_2}:{N_2}}$) and the random minima ($X_{{1}:{N_1}}$ and $X_{{1}:{N_2}}$), when the sample sizes are different and random (positive).

\section{Introduction}

Let $\{X_{1},X_{2},\ldots\}$ be a sequence of independent and identically distributed (i.i.d) random variables and $N$ be a positive integer-valued random variable independent of $X_{i}'$s. Also, let $X_{1:N}=min\{X_1,\ldots,X_{N}\}$ and $X_{N:{N}}=max\{X_1,\ldots,X_{N}\}$ be the random minima and maxima corresponding to the considered random variables.
For non-negative random variables $X_i,$ the random maxima $X_{N:N}$ naturally appears in reliability theory as the lifetime of a parallel system of identical components with lifetimes $X_1,\ldots, X_{N}.$ Also, in actuarial science, in a particular time period, the claims received by an insurer should be a sample of random size, and then $X_{N:N}$ denotes the largest claim amount in that period. The random variable $X_{1:{N}}$ arises naturally in transportation theory as the accident-free distance of a shipment of, say, explosives, where $N$ of them are defectives which may explode and cause an accident after $X_1,\ldots, X_{N}$ miles, respectively. Thus, comparing random maxima or minima stochastically is of great interest from a practical point of view. In this direction, only few works exist in the literature. Among them, \cite{shaked1997} considered two different i.i.d. samples $X_1,X_2,\ldots$ and $Y_1,Y_2,\ldots$ having the same (random) sample size ($N$) and showed that if $X_{i}\leq_{st}Y_{i},$ for $i=1,\ldots,N$ then $X_{1:{N}}\leq_{st}Y_{1:{N}}$ and $X_{N:{N}}\leq_{st}Y_{N:{N}}.$ Moreover, they showed that 
if Laplace transform and Laplace transform ratio orders hold between $N_1$ and $N_2,$ then the usual 
stochastic, hazard rate, reversed hazard rate and likelihood ratio orders are hold between 
$X_{1:{N_1}}$, $X_{1:{N_2}}$ and $X_{N_{1}:{N_1}}$, $X_{N_{2}:{N_2}}$, where $N_1$ and $N_2$ are two positive integer-valued random variables, independent of $X_{i}'$s. Subsequently, \cite{bartoszewicz2001} proved that if $X_1$ is smaller than $Y_1$ according to convex, star and super-additive orders, then $X_{1:N}$ and $X_{N:N}$ are smaller than $Y_{1:N}$ and $Y_{N:N},$ respectively. Then, \cite{Li2004} proved that right spread and increasing convex orderings hold between $X_{N:N}$ and $Y_{N:N},$ under some certain conditions. Also, they showed that total time on test transform 
and increasing concave orderings hold between $X_{1:N}$ and $Y_{1:N}$. Subsequently, \cite{Ahmad2007} presented reversed preservation property of right spread and total time on test transform orders under random minima and maxima. Recently, \cite{chowdhury2024} established several interesting results based on comparison of random minima and maxima from a random number of non-identical random variables. To get an overview of the results on stochastic comparisons of random maxima and minima, we refer to \cite{Nanda2008}. 

 {A random variable is said to follow Kumaraswamy’s distribution if its cumulative distribution function (CDF) is given by $$F(x)=1-(1-x^a)^b ,~x\in(0,1)$$ with two shape parameters $a>0$ and $b>0.$ This novel two-parameter distribution was first introduced by \cite{Kum1980} to model hydrological data. In the context of hydrological applications, the well known  probability distributions such as the Beta, Normal, Log-Normal, and Student-t often fail to provide adequate fit. For this reason, Kumaraswamy distribution has
received considerable attention in hydrology and related areas (see \cite{Sundar1989}, \cite{Fletcher1996} and  \cite{Seifi2000}).  A random variable $X$ is said to follow Kumaraswamy generalized (denoted by $X\sim\text{Kw-G}(x,\alpha,\gamma;G)$) distribution if its distribution function is given by 
\begin{equation}\label{KW-G}
F(x,\alpha,\gamma)=1-(1-G^{\alpha}(x))^{\gamma},~\alpha,~\gamma>0,~ x\in (-\infty, \infty),
\end{equation}
where $G(\cdot)$ is said to be the baseline distribution while $\alpha$ and $\gamma$ are two positive parameters whose role is to govern skewness and varying tail weights. {This distribution was first introduced by \cite{Cordeiro2011} to generate new families of distributions.} In \eqref{KW-G}, if we replace $G(x)$ by the distribution function of normal, Weibull and gamma, then we can obtain $\text{Kw}$-normal, $\text{Kw}$-Weibull and $\text{Kw}$-gamma distributions, respectively. In survival analysis and reliability engineering, in many different real life problems, censored data are common due to incomplete information in the observation of survival times. In this context, $\text{Kw-G}$ distribution is flexible and powerful distribution to model such censored data. It allows a baseline distribution (like Weibull, exponential, or Log-logistic) to capture more complex behaviors in real-world data possessing asymmetry, heavy tails and varied hazard functions. Recently, $\text{Kw-G}$ family of distributions has received considerable attention in the literature due to its wide applications (see \cite{Mohammed2014} and \cite{Pavlov2018}). In this regard, it will also be of interest to study stochastic properties of their order statistics. \cite{kundu2018w} developed different stochastic properties of minimum order statistics for $\text{Kw-G}$ model using majorization and related orders when the observations are independent. \cite{kayal2019kw} considered the same model and established ordering results for series as well as parallel systems based on multivariate chain majorization and vector chain majorization orders. Recently, \cite{kundu2021_kwg} established several comparison results concerning maximums of two independent heterogeneous $\text{Kw-G}$ random variables when each of the units in the sample experience a random shock.
Various ordering results for this model has been discussed in the literature by considering independent and heterogeneous variables for the case of fixed sample size. We therefore concentrate on dependent observations in the context of a random sample size thus differing from all previous works.}

 In many real-life situations, lifetimes of a system's components are dependent. Due to the sharing of common workloads such as heat and tasks, dependence between components may occur. For example, in the 1960's space development program, dependent failure was observed due sharing of heat. Because of inside temperature, the second guidance computer failed after few minutes following the failure of the first one during the time of reentry of Gemini spacecraft (see \cite{OTA2017376}). {As another example, we have the traditional two-units series
system subjected to shocks coming from three shock sources for its failure. Shock from
the first and second sources affect the first and second units, respectively, while
the last shock from the third source affects both units simultaneously. It is evident that the causes of failure are dependent in this system. Therefore, comparing the lifetimes of systems { with} dependent components has attracted a great deal of attention in the reliability  and insurance analysis literature. In this regard, the concept of Archimedean copulas is a versatile tool to model dependent observations. In particular, in insurance analysis, Archimedean copula provides
sufficient additional structure to allow strengthening as well as clarification of findings
on the nature of portfolio allocation vectors and of its comparative statics (see \cite{Hennessy2002} and \cite{Kularatne2021}).

 In our work, we use Archimedean copula functions to explore the dependence structure between the random number of observations. A copula is a function which connects the marginal distributions to reconstruct the joint distribution. Using the copula approach, one can easily model different relationships that can exist in different ranges of behaviour. In particular, to analyze the dependence structure, there are many advantages to use an Archimedean copula function, and these are as follows:
\begin{itemize}
    \item [(i)] Archimedean copulas enable us to model the marginal behaviour and the dependence structure separately;
 \item [(ii)]  The Archimedean copula function gives us the information about the degree
of dependence and also structure of dependence. It allows for asymmetric dependence as a linear correlation  does not always provide information about tail dependence;
\item [(iii)] To capture tail dependence, some specific Archimedean copulas having such interesting property (for example, Gumbel copula (upper tail-dependence) or Clayton copula (lower tail-dependence)) have been used in finance, insurance, and reliability problems;
\item [(iv)] To develop efficient estimation of model parameters, Archimedean copulas are good to use as they are relatively easy to simulate from when we are dealing with dependent observations.
\end{itemize}
   }

 In this article, we consider $\{X_{1},\ldots,X_{N_1}\}$ and  $\{Y_{1},\ldots,Y_{N_2}\}$ {as} two different interdependent and heterogeneous random variables, where $X_{i}\sim \text{Kw-G}(x, \alpha_{i}, \gamma_{i};G)$ and $Y_{i}\sim \text{Kw-G}(x, \beta_{i}, \delta_{i};H),$ with $G(\cdot)$ and $H(\cdot)$ being the baseline distributions, and $N_1$ and $N_2$ being two positive integer-valued random variables, independently of $X_{i}'$s  and $Y_{i}'$s, respectively. Several stochastic orders such as usual stochastic, hazard rate, reversed hazard rate, dispersive and likelihood ratio orders are considered here for the comparison results between the random maxima and minima when the sample sizes are different and random (positive). Here, we use Archimedean copula to model the dependence between the component lifetimes.
 
 In Section \ref{s1}, we recall some important definitions of stochastic orders and majorizations. The main results are presented in Section \ref{s21.}. In this section, we establish the comparison results between two random maxima and minima in { term} of usual stochastic, hazard rate, reversed hazard rate, dispersive and likelihood orders, when the variables are heterogeneous, {independent/interdependent} and follow $\text{Kw-G}$ family of distributions assembled with Archimedean (survival) copula { with} different generators and the sample sizes are positive random variables. {Moreover, some applications of our results in biostatistics and transportation theory are presented in Section \ref{app}}. Finally, some concluding remarks are made in Section \ref{c}.
 
 In this article, we concentrate on random variables defined on $(0,\infty)$.
The terms ‘increasing’ and ‘decreasing’ are used in nonstrict sense. Moreover, `$\overset{sign}{=}$' is used to
denote both sides of as equality having the same sign.
\section{Preliminaries}\label{s1}
Here, we review some important definitions and well-known concepts involving the notion of majorization and stochastic order. Let  $\boldsymbol{x} =
\left(x_{1},\cdots,x_{n}\right)$ and $\boldsymbol{y} =
\left(y_{1},\ldots,y_{n}\right)$ be two $n$ dimensional vectors such that $\boldsymbol{x}~,\boldsymbol{y}\in\mathbb{A}$. Here, $\mathbb{A} \subset \mathbb{R}^{n}$ and $\mathbb{R}^{n}$ is
an $n$-dimensional Euclidean space. Also, consider the order coordinates of the vectors $\boldsymbol{x}$ and $\boldsymbol{y}$ as  $x_{1:n}\leq \cdots \leq x_{n:n}$ and
$y_{1:n}\leq\cdots \leq y_{n:n},$ respectively. A vector $\boldsymbol{x}$ is said to be majorized by another vector $\boldsymbol{y}$ (denoted by $\boldsymbol{x}\preceq^{m} \boldsymbol{y}$) if, for each $l=1,\ldots,n-1$, $\sum_{i=1}^{l}x_{i:n}\geq \sum_{i=1}^{l}y_{i:n}$ and	$\sum_{i=1}^{n}x_{i:n}=\sum_{i=1}^{n}y_{i:n}$ hold. A vector $\boldsymbol{x}$ is said to be weakly submajorized by another vector $\boldsymbol{y}$ (denoted by $\boldsymbol{x}\preceq_{w} \boldsymbol{y}$) if, for each $l=1,\ldots,n$, $\sum_{i=l}^{n}x_{i:n}\leq \sum_{i=l}^{n}y_{i:n}$ holds. A vector $\boldsymbol{x}$ is said to be weakly supermajorized by another vector $\boldsymbol{y}$ (denoted by $\boldsymbol{x}\preceq^{w} \boldsymbol{y}$) if, for each $l=1,\ldots,n$, $\sum_{i=1}^{l}x_{i:n}\geq \sum_{i=1}^{l}y_{i:n}$ holds. It is important to note that $\boldsymbol{x}\preceq^{m} \boldsymbol{y}$ implies both  $\boldsymbol{x}\preceq_{w} \boldsymbol{y}$ and  $\boldsymbol{x}\preceq^{w} \boldsymbol{y}.$ But, the converse is not always true. For more details on majorization order and their applications, we refer \cite{Marshall2011}. 

Let $X_1$ and $X_2$ be two univariate random variables with
 density functions (PDFs) $f_{X_1}$ and $f_{X_2}$, distribution functions (CDFs) $F_{X_1}$ and $F_{X_2}$, survival functions $\bar
F_{X_1}$ and $\bar F_{X_2}$, hazard rate functions  $r_{X_1}=f_{X_1}/\bar
F_{X_1}$ and $ r_{X_2}=f_{X_2}/
\bar{F}_{X_2},$ and reversed hazard rate functions $\tilde{r}_{X_1}=f_{X_1}/ F_{X_1}$ and $\tilde r_{X_2}=f_{X_2}/
F_{X_2}$, respectively. Also, let $l_{X_1}$ and $l_{X_{2}}$ and $u_{X_1}$ and $u_{X_{2}}$  be the left and right end points of the supports of $X_{1}$ and $X_{2},$ respectively. Then, a random variable $X_1$ is said to be smaller than $X_2$ in the usual stochastic order (denoted by $X_1\leq_{st}X_2$) if	$\bar F_{X_1}(x)\leq\bar F_{X_2}(x)$, for all $x.$ A random variable $X_1$ is said to be smaller than $X_2$ in the hazard rate order (denoted by $X_1\leq_{hr}X_2$)
		if $\bar{F}_{X_2}/\bar{F}_{X_1}$ is increasing in x, for all $x\in(-\infty,\max(u_{X_1}, u_{X_2})).$ If the hazard rate exists, we can say $X_1\leq_{hr}X_2$ if and only if $ r_{X_1}(x)\geq  r_{X_2}(x)$, for all $x$.  A random variable $X_1$ is said to be smaller than $X_2$ in the  reversed hazard rate order (denoted by $X_1\leq_{rh}X_2$) if ${F}_{X_2}/{F}_{X_1}$ is increasing in x, for all $x\in(-\infty,\max(u_{X_1}, u_{X_2})).$ If the reversed hazard rate exists, we can say $X_1\leq_{rh}X_2$ if and only if $ \tilde{r}_{X_1}(x)\leq  \tilde{r}_{X_2}(x)$, for all $x$. A random variable $X_1$ is said to be smaller than $X_2$ in the likelihood ratio order (denoted by $X_1\leq_{lr}X_2$) if $f_{X_2}(x)/f_{X_1}(x)$ is increasing in $x$ for all $x\in(-\infty,\max(u_{X_1}, u_{X_2}))$. A random variable $X_1$ is said to be smaller than $X_2$ in the
   dispersive order (denoted by $X_{1}\le_{disp}X_2{}$) if
		$F^{-1}_{X_{1}}(\beta)
		-F^{-1}_{X_{1}}(\alpha)\le F^{-1}_{X_{2}}(\beta)
		-F^{-1}_{X_{2}}(\alpha)\text{ whenever }0<\alpha\leq\beta<1,$ where $F^{-1}_{X_{1}}(x)$ and $F^{-1}_{X_{2}}(x)$ are the right-continuous inverses of $F_{X_{1}}(x)$ and $F_{X_{2}}(x),$ respectively. It is known that $\leq_{lr}\Rightarrow\leq_{hr}(\leq_{rh})\Rightarrow\leq_{st}.$ Using these notions, we can easily derive useful bounds for the survival functions as well as hazard rate functions that will be useful in reliability theory and survival analysis (see \cite{pledger1971comparisons}). It is also important to note that dispersive ordering has been used by several authors to obtain bound for some inferential problems (see \cite{Jeon2006}). We refer to \cite{shaked2007stochastic} for a detailed discussion on stochastic orderings. 

	The idea of Archimedean copula is one of the tools to model dependent data. Let $F$ and $\bar F$ be the joint distribution function and joint survival function of a random vector $\boldsymbol{X}=(X_1,\cdots,X_n)$. Also, let there exist some functions $C(\boldsymbol{v}):[0,1]^n\rightarrow [0,1]$ and  
	$\hat {C}(\boldsymbol{v}):[0,1]^n\rightarrow [0,1]$ such that for all $ x_i,~i\in \mathcal I_n, $ where $\mathcal I_n$ is the index set, $ F(x_1,\cdots,x_n)=C(F_1(x_1),\cdots,F_n(x_n))$ and
	$\bar{F}(x_1,\cdots,x_n)=\hat{C}(\bar{F_1}(x_1),\cdots,\bar{F_n}(x_n))$ hold, where $\boldsymbol{v}=(v_1,\cdots,v_n)$. Then, $C(\boldsymbol{v})$ and $\hat{C}(\boldsymbol{v})$ are said to be the  copula and survival copula of $\boldsymbol{X}$, respectively. Here, $F_1,\cdots,F_n$ and $\bar{F_1},\cdots,\bar{F_n}$ are the univariate marginal distribution functions and survival functions of the random variables $X_1,\cdots,X_n$, respectively.
	Let $\psi:[0,\infty)\rightarrow[0,1]$ be a non-increasing and continuous function, satisfying $\psi(0)=1$ and $\psi(\infty)=0.$ Also, let $\phi={\psi}^{-1}=\text{sup}\{x\in \mathcal R:\psi(x)>v\}$ be the right continuous inverse. Further, suppose $\psi$ satisfies the conditions (i)$(-1)^i{\psi}^{i}(x)\geq 0,~ i=0,1,\cdots,d-2,$ and  (ii)$(-1)^{d-2}{\psi}^{d-2}$ is non-increasing and convex. These imply that the generator $\psi$ is $d$ monotone. Then, a copula $C_{\psi}$ is said to be an Archimedean copula if it can be written as $C_{\psi}(v_1,\cdots,v_n)=\psi({\psi^{-1}(v_1)}+\cdots+\psi^{-1}(v_n)),~\text{ for all } v_i\in[0,1],~i\in\mathcal{I}_n.$ For extensive and comprehensive details on Archimedean copulas, one may refer to \cite{nelsen2006introduction}.
 Denote
$(i)~\mathcal{D}_{+}=\{(x_1,\cdots,x_n):x_{1}\geq
x_{2}\geq\cdots\geq x_{n}>0\},$ $(ii)~\mathcal{E}_{+}=\{(x_1,\cdots,x_n):0<x_{1}\leq
x_{2}\leq\cdots\leq x_{n}\},$ $(iii) \boldsymbol{1}_{n}=(1,\cdots,1),$ and $(iv) h'( z)=\frac{d h(z)}{d z}.$ These notations will be consistently used henceforth.

\section{Main results}\label{s21.}
For baseline distributions $G(\cdot)$ and $H(\cdot),$ let $\{X_{1},\cdots,X_{n}\}$ and $\{Y_{1},\cdots,Y_{n}\}$ be two sets of $n$ interdependent variables coupled with Archimedean (survival) copula { with} generators $\psi_1$ and $\psi_2,$ where, for $i=1,\cdots,n,$ $X_{i}$'s follow $\text{Kw-G}$ distribution with parameters $\alpha_{i}$, $\gamma_{i},$ and $Y_{i}$'s follow $\text{Kw-G}$ distribution with parameters $\beta_{i}$ and $ \delta_{i}.$ We denote $\boldsymbol{X}\sim \text{Kw-G}(x,\boldsymbol{\alpha},\boldsymbol{\gamma};G)$ and 
$\boldsymbol{Y}\sim \text{Kw-G}(x,\boldsymbol{\beta},\boldsymbol{\delta};H).$
Further, let $X_{n:n}$ and $X_{1:n},$ $Y_{n:n}$ and $Y_{1:n}$ be the largest and the smallest order statistics corresponding to the observations $\{X_{1},\cdots,X_{n}\}$ and $\{Y_{1},\cdots,Y_{n}\},$ respectively. Also, let $N_1$ and $N_{2}$ be two positive integer-valued discrete random variables independently of $X_i'$s and $Y_i's$, respectively. Then, under the above assumptions, the distribution function of $X_{N_{1} : N_{1}}$ and the reliability function of $X_{1 : N_{1}}$ can be expressed as
\begin{equation*}
  {F}_{X_{{N_{1}}:{N_{1}}}}(x) =\sum_{m=1}^{n}P({X_{N_{1}:{N_{1}}}< x}|N_{1}=m)P(N_{1} =m)
    =\sum_{m=1}^{n}P({X_{{m}:{m}}< x})P(N_{1} =m),
\end{equation*}
\begin{equation*}
    \bar{F}_{X_{1:{N_{1}}}}(x) =\sum_{m=1}^{n}P({X_{1:{N_{1}}}> x}|N_{1}=m)P(N_{1} =m)
    =\sum_{m=1}^{n}P({X_{{1}:{m}}> x})P(N_{1} =m),
\end{equation*}
where
\begin{equation*}
    P({X_{{n}:{n}}< x})=\psi_{1}\left(\sum\limits_{i=1}^{n}\phi_{1}\{1-\left(1-G^{\alpha_{i}}(x)\right)^{\gamma_{i}}\}\right)  \text{  and  } P({X_{{1}:{n}}> x})=\psi_{1}\left(\sum\limits_{i=1}^{n}\phi_{1}\{\left(1-G^{\alpha_{i}}(x)\right)^{\gamma_{i}}\}\right) .
 \end{equation*}
 Similarly, we can present the distribution function of $X_{N_{2} : N_{2}}$ and the reliability function of $X_{1 : N_{2}}.$
Now, the following assumptions will be made throughout this section.
\begin{assumption}\label{ass1}
	Suppose $\{X_{1},\cdots,X_{n}\}~[\{Y_{1},\cdots,Y_{n}\}]$ are $n$ interdependent variables coupled with Archimedean (survival) copula { with} generator $\psi_1~[\psi_2],$ where, for $i=1,\cdots,n,$ $X_{i}\sim \text{Kw-G}(x,\alpha_{i},\gamma_{i};G)$
	$[Y_{i}\sim \text{Kw-G}(x,\beta_{i},\delta_{i};H)]$ and $G~[H]$ is the baseline distribution function. Further, let $N_1~[N_2]$ be a positive integer-valued random variable independently of $X_{i}'~[Y_{i}']$. Also, let us denote $r_{g}(x)=\frac{G'(x)}{\bar{G}(x)}$ and $r_{h}(x)=\frac{H'(x)}{\bar{H}(x)}.$
	
\end{assumption}
We start this section by the following theorem which states that under some sufficient conditions, if $\boldsymbol{\alpha}\succeq^{w}\boldsymbol{\beta},$ then the random maxima $X_{{N_1}:{N_1}}$ is stochastically larger than $Y_{{N_2}:{N_2}}.$
\begin{theorem}\label{th1}
	Let Assumption \ref{ass1} hold with $\boldsymbol{\alpha},~\boldsymbol{\beta}\in\mathcal{E}_+(\mathcal{D}_+),$ $\boldsymbol{\gamma}=\boldsymbol{\delta}(\geq \boldsymbol{1}_{n}),$ $G\leq H$ . Also, suppose $\phi_{2}\circ\psi_{1}$ is super-additive and $\psi_1$ or $\psi_2$ is log-concave. Then, $\boldsymbol{\alpha}\succeq^{w}\boldsymbol{\beta}\Rightarrow X_{N_{1}:N_{1}}\geq_{st} Y_{N_{2}:N_{2}},$ whenever $N_{1}\leq_{st}N_{2}$.
\end{theorem}
\begin{proof}
The main step in proving the theorem is to establish that
$$\boldsymbol{\alpha}\succeq^{w}\boldsymbol{\beta}\Rightarrow P(X_{n:n}<x)\leq P(Y_{n:n}<x)\Rightarrow \mathcal{L} (\boldsymbol{\alpha}, \boldsymbol{\gamma},\psi_{1};G)\leq \mathcal{L}(\boldsymbol{\beta}, \boldsymbol{\gamma},\psi_{2};H),$$ where
$\mathcal{L}(\boldsymbol{\alpha}, \boldsymbol{\gamma},\psi_{1};G)=\psi_{1}\left(\sum\limits_{i=1}^{n}\phi_{1}\{1-\left(1-G^{\alpha_{i}}(x)\right)^{\gamma_{i}}\}\right)$
	and
	$\mathcal{L}(\boldsymbol{\beta}, \boldsymbol{\gamma},\psi_{2};H)=\psi_{2}\left(\sum\limits_{i=1}^{n}\phi_{2}\{1-\left(1-H^{\beta_{i}}(x)\right)^{\gamma_{i}}\}\right).$ Then, using the fact that $N_{1}\leq_{st}N_{2}$, we will complete the remaining part. From the given conditions, 
	$$\mathcal{L}(\boldsymbol{\alpha}, \boldsymbol{\gamma},\psi_{1};G)\leq\mathcal{L}(\boldsymbol{\alpha}, \boldsymbol{\gamma},\psi_{2};G)\leq\mathcal{L}(\boldsymbol{\alpha}, \boldsymbol{\gamma},\psi_{2};H).$$
	Therefore, we only need to establish that	$$\boldsymbol{\alpha}\succeq^{w}\boldsymbol{\beta}\Rightarrow \mathcal{L}(\boldsymbol{\alpha}, \boldsymbol{\gamma},\psi_{2};H)\leq \mathcal{L}(\boldsymbol{\beta}, \boldsymbol{\gamma},\psi_{2};H).$$
According to Theorem $A.8$ of \cite{Marshall2011}, the above inequality will be fulfilled if we show that $\mathcal{L}(\boldsymbol{\alpha}, \boldsymbol{\gamma},\psi_{2};H)$ is decreasing and Schur-convex in $\boldsymbol{\alpha}\in\mathcal{E}+(\mathcal{D}_{+})$. To this aim, consider the case when $\boldsymbol{\alpha}\in \mathcal{E}_{+}.$ One can also prove the other part in a similar fashion. Taking partial derivative of $\mathcal{L}(\boldsymbol{\alpha}, \boldsymbol{\gamma},\psi_{2};H)$ with respect to $\alpha_i$, we have
	for $1\leq i\leq j\leq n,$
	\begin{equation}\label{eq-sconcave-max-alpha_1}
		\frac{\partial \mathcal{L}(\boldsymbol{\alpha}, \boldsymbol{\gamma},\psi_{2};H)}{\partial\alpha_{i}}-\frac{\partial \mathcal{L}(\boldsymbol{\alpha}, \boldsymbol{\gamma},\psi_{2};H)}{\partial\alpha_{j}}\Rightarrow\mathcal{L}'(\boldsymbol{\alpha}, \boldsymbol{\gamma},\psi_{2};H)\left[\frac{s_{i}\chi^*(\alpha_i,\gamma_i)}{\psi'_2(\phi_2\{s_{i}\})}-\frac{s_{j}\chi^*(\alpha_j,\gamma_j)}{\psi'_2(\phi_2\{s_{j}\})}\right],
	\end{equation}
	where $s_{i}=1-\left(1-G^{\alpha_{i}}(x)\right)^{\gamma_{i}}$ and $\chi^*(\alpha_{i},\gamma_{i})=\frac{\gamma_{i}\ln G(x)G^{\alpha_{i}}(x)(1-G^{\alpha_{i}(x)})^{\gamma_{i}-1}}{1-(1-G^{\alpha_{i}}(x))^{\gamma_i}}.$ Note that $s_i$ is decreasing in $\alpha_i$, increasing in $\gamma_{i}$ and by Lemma $3.1$ of \cite{balakrishnan2015stochastic}, $\chi^*(\alpha_i,\gamma_i)$ is decreasing in $\alpha_{i}$ and increasing in $\gamma_{i},$ for $i=1,\ldots,n.$ Therefore, by the given conditions, we can say that \eqref{eq-sconcave-max-alpha_1} is non-positive. Also, $\mathcal{L}(\boldsymbol{\alpha}, \boldsymbol{\gamma},\psi_{2};H)$ is increasing in $\boldsymbol{\alpha}\in \mathcal{E}_{+}$ by Lemma $3.3$ of \cite{kundu2016some}. 
 Now, $N_{1}\leq_{st}N_{2}$ yields
\begin{align}
    {F}_{X_{{N_{1}}:{N_{1}}}}(x) &=\sum_{m=1}^{n}P({X_{{m}:{m}}< x})P(N_{1} =m)\nonumber\\
    &\leq\sum_{m=1}^{n}P({X_{{m}:{m}}< x})P(N_{2} =m)\nonumber\\
    &\leq\sum_{m=1}^{n}P({Y_{{m}:{m}}< x})P(N_{2} =m)\nonumber\\
    &\leq {F}_{Y_{{N_{2}}:{N_{2}}}}(x),
\end{align}
which completes the proof of the theorem.
\end{proof}

Next, we consider a counterexample to show that the condition $\psi_1$ or $\psi_2$ is log-concave plays a crucial role for the {result} in Theorem \ref{th1} to hold.

\begin{counterexample}\label{cex1.1}
Let, $X_i \sim \text{Kw-G}(x,\alpha_i,\gamma_i;G)$ and $Y_i \sim \text{Kw-G}(x,\beta_{i},\gamma_i;H),$ for $i=1,2,3,4,5.$ Set $(\gamma_1,\gamma_2 ,\gamma_3,\gamma_4,\gamma_5)=(1.2, 1.5, 1.9,2,2.1)$, $(\alpha_1, \alpha_2,\alpha_3,\alpha_4,\alpha_5)=(1.1, 0.001,0.0001,0.00001,0.00001)$ and $(\beta_1, \beta_2, $
$\beta_3,\beta_4,\beta_5)=(2.1, 3.001,5.0001,0.001,0.0001)$. Clearly, $(\alpha_1, \alpha_2,\alpha_3,\alpha_4,\alpha_5)\succeq^{w}(\beta_1, \beta_2, \beta_3,\beta_4,\beta_5)$. Let $G(x)=e^{-\frac{1}{x}}$ and $H(x)=1-e^{-x},$ for $x>0.$ Also, let $(X_{1},X_{2},X_{3})$ be selected with probability $P(N_{1}=3)=p(3)=1/5,$ $(X_{1},X_{2},X_{3},X_{4})$ be selected with probability $P(N_{1}=4)=p(4)=2/5$ and, $(X_{1},X_{2},X_{3},X_{4},X_{5})$ be selected with probability $P(N_{1}=5)=p(5)=2/5;$ further let $(Y_{1},Y_{2},Y_{3})$ be selected with probability $P(N_{2}=3)=p(3)=1/5,$ $(Y_{1},Y_{2},Y_{3}, Y_{4})$ be selected with probability $P(N_{2}=4)=p(4)=3/5$ and $(Y_{1},Y_{2},Y_{3}, Y_{4},Y_{5})$ be selected with probability $P(N_{2}=5)=p(5)=1/5.$ Here, $N_{1}\leq_{st}N_{2}.$ Let us use the Gumbel copula { with} generators $\psi_1(x) = e^{-x^{\frac{1}{a}}}$ and $\psi_2 (x)= e^{-x^{\frac{1}{b}}},~x>0,$ with $a=2$ and $b=2.1$. Clearly, all the conditions are satisfied except $\psi_1$ and $\psi_2$ are log-concave. Figure $1(a)$ displays that the difference between ${F}_{X_{{N_{1}}:{N_1}}}(x)$ and ${F}_{Y_{{N_{2}:N_{2}}}}(x)$ crosses $x$ axis, which violates the result in Theorem \ref{th1}.

\begin{figure}[h]
		\begin{center}
			\subfigure[]{\label{c3}\includegraphics[height=2.0in]{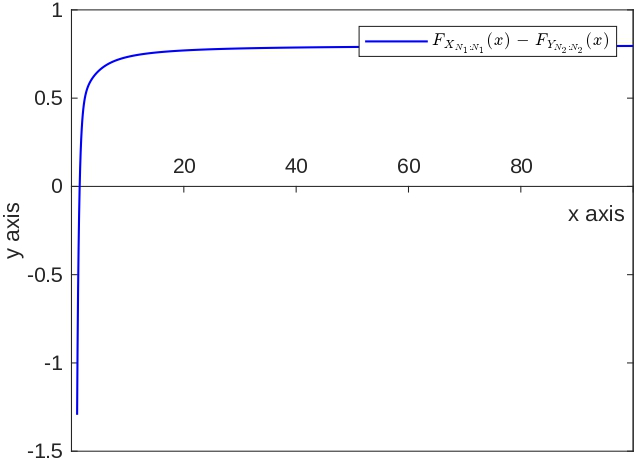}}
			\subfigure[]{\label{c3.}\includegraphics[height=2.0in]{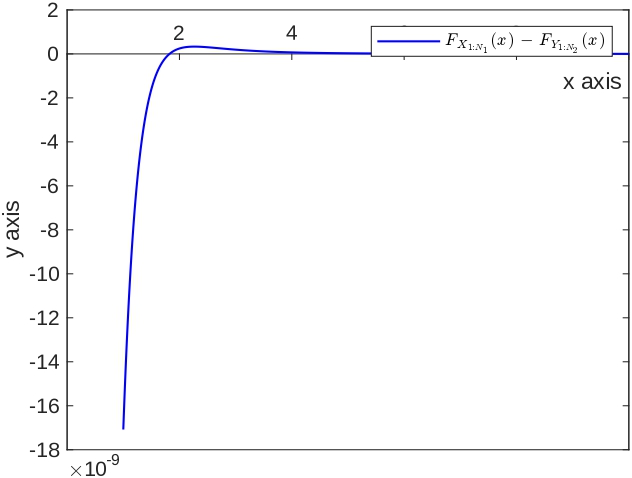}}
			\caption{(a) Plot of $\bar{F}_{X_{{N_1}:{N_{1}}}}(x)-\bar{F}_{Y_{{N_2}:{N_2}}}(x)$ in Counterexample \ref{cex1.1} when $\psi_1$ and $\psi_2$ are both log-concave; (b) Plot of $\bar{F}_{X_{{1}:{N_{1}}}}(x)-\bar{F}_{Y_{{1}:{N_2}}}(x)$ in Counterexample \ref{cex1.3}.}
		\end{center}
	\end{figure}
 \end{counterexample}
As in Theorem \ref{th1}, the following theorem shows that if $\boldsymbol{\delta}\succeq^{w}\boldsymbol{\gamma},$ then the opposite inequality holds between the two random maxima $X_{{N_1}:{N_1}}$ and $Y_{{N_2}:{N_2}}.$ 

\begin{theorem}\label{th2}
	Let Assumption \ref{ass1} hold with $\boldsymbol{\alpha}=\boldsymbol{\beta},$ $\boldsymbol{\alpha},~\boldsymbol{\gamma},~\boldsymbol{\delta}\in\mathcal{E}_+(\mathcal{D}_+)$ and $G\leq H$. Also, suppose $\phi_{2}\circ\psi_{1}$ is super-additive and $\psi_1$ or $\psi_2$ is log-concave.Then, $\boldsymbol{\delta}\succeq^{w}\boldsymbol{\gamma}\Rightarrow X_{{N_{1}}:{N_{1}}}\leq_{st} Y_{{N_{2}}:{N_{2}}},$ whenever $N_{1}\geq_{st}N_{2}$.
\end{theorem}
In the above two theorems, we have presented some conditions for comparing two random maxima. The following theorem provides a comparison between two random minima when $\boldsymbol{\gamma}\succeq_{w}\boldsymbol{\delta}.$

\begin{theorem}\label{th3}
	Let Assumption \ref{ass1} hold with $\boldsymbol{\alpha}=\boldsymbol{\beta},$ $\boldsymbol{\alpha},~\boldsymbol{\gamma},~\boldsymbol{\delta}\in\mathcal{E}_+(\mathcal{D}_+)$ and $G\geq H.$ Also, suppose $\phi_{2}\circ\psi_{1}$ is super-additive and $\psi_1$ or $\psi_2$ is log-convex. Then, $\boldsymbol{\gamma}\succeq_{w}\boldsymbol{\delta}\Rightarrow X_{1:{N_{1}}}\leq_{st} Y_{1:{N_{2}}},$ whenever $N_{1}\leq_{st}N_{2}$.
\end{theorem}
\begin{proof}
Let us denote $\mathcal{L}_{1}(\boldsymbol{\alpha}, \boldsymbol{\gamma},\psi_{1};G)=\psi_{1}\left(\sum\limits_{i=1}^{n}\phi_{1}\{\left(1-G^{\alpha_{i}}(x)\right)^{\gamma_{i}}\}\right)$
	and
	$\mathcal{L}_{1}(\boldsymbol{\alpha}, \boldsymbol{\delta},\psi_{2};H)=\psi_{2}\left(\sum\limits_{i=1}^{n}\phi_{2}\{\left(1-H^{\alpha_{i}}(x)\right)^{\delta_{i}}\}\right).$
	Using the given conditions, we then have  
	$$\mathcal{L}_{1}(\boldsymbol{\alpha}, \boldsymbol{\gamma},\psi_{1};G)\leq\mathcal{L}_{1}(\boldsymbol{\alpha}, \boldsymbol{\gamma},\psi_{2};G)\leq\mathcal{L}_{1}(\boldsymbol{\alpha}, \boldsymbol{\gamma},\psi_{2};H).$$
	According to Theorem $A.8$ of \cite{Marshall2011}, to establish the required result, it is enough to show that $\mathcal{L}_1(\boldsymbol{\alpha}, \boldsymbol{\gamma},\psi_{2};H)$ is decreasing and Schur-concave in $\boldsymbol{\gamma}\in \mathcal{E}_{+}(\mathcal{D}_{+}).$ Consider the case when $\boldsymbol{\gamma}\in\mathcal{E}_{+}.$ The other case can be established similarly.
	The partial derivative of $\mathcal{L}_1(\boldsymbol{\alpha}, \boldsymbol{\gamma},\psi_{2};H)$ with respect to $\gamma_i$ is given by,
	for $1\leq i\leq j\leq n,$
	\begin{equation}\label{eq-sconcave-max-alpha}
		\frac{\partial \mathcal{L}_1(\boldsymbol{\alpha}, \boldsymbol{\gamma},\psi_{2};H)}{\partial\gamma_{i}}-\frac{\partial \mathcal{L}_1(\boldsymbol{\alpha}, \boldsymbol{\gamma},\psi_{2};H)}{\partial\gamma_{j}}=\mathcal{L}'_1(\boldsymbol{\alpha}, \boldsymbol{\gamma},\psi_{2};H)\left[\frac{s_{1i}\chi^*_{1}(\alpha_i,\gamma_i)}{\psi'_2(\phi_2\{s_{1i}\})}-\frac{s_{1j}\chi^*_{1}(\alpha_j,\gamma_j)}{\psi'_2(\phi_2\{s_{1j}\})}\right],
	\end{equation}
	where $s_{1i}=\left(1-G^{\alpha_{i}}(x)\right)^{\gamma_{i}}$ and $\chi^*_{1}(\alpha_{i})={\ln(1-G^{\alpha_{i}}(x))}.$ Note that $s_{1i}$ is increasing in $\alpha_i$, decreasing in $\gamma_{i}$, which implies $\chi^*_{1}(\alpha_i)$ is increasing in $\alpha_{i},$ for $i=1,\ldots,n.$ Therefore, by the given conditions and Lemma $3.3$ of \cite{kundu2016some}, we have \eqref{eq-sconcave-max-alpha} to be non-negative and $\mathcal{L}_1(\boldsymbol{\alpha}, \boldsymbol{\gamma},\psi_{2};H)$ is decreasing in $\boldsymbol{\gamma}\in \mathcal{E}_{+}.$ The remaining part of the proof can be completed by the same lines as in the proof of Theorem \ref{th1}. 
\end{proof}
\begin{remark}
In Theorem \ref{th3}, if we consider $\bm{\alpha}=\bm{\beta}=\bm{1}_{n},$ $N_{1}$ and $N_2$ are not random variables and remove the condition $G\geq H,$ then this theorem is an extension of Part $(1)$ of Theorem $4.1$ of \cite{fang2016stochastic}.
\end{remark}
Next, we present another counterexample to show that if we replace the condition ``$G\leq H$'' in Theorem \ref{th3} by ``$G\geq H$'', then the above theorem may not be true.
\begin{counterexample}\label{cex1.3}
Let, $X_i \sim \text{Kw-G}(x,\alpha_i,\gamma_i;G)$ and $Y_i \sim \text{Kw-G}(x,\alpha_{i},\delta_i;H),$ for $i=1,2,3,4,5.$ Further, let $(\gamma_1,\gamma_2 ,\gamma_3,\gamma_4,\gamma_5)=(3.9,3.1,2.9,2.8,2.1)$, $(\delta_{1},\delta_{2},\delta_{3},\delta_{4},\delta_{5})=(3.2,2.6,1.9,1.2,1)$ and $(\alpha_1,\alpha_2,\alpha_3,\alpha_4,\alpha_5)=(\beta_1, \beta_2,\beta_3,\beta_4,\beta_5)=(0.01,7,9,9.1,9.12)$. Clearly, $(\gamma_1,\gamma_2, \gamma_3,\gamma_4,\gamma_5)\succeq_{w}(\delta_{1},\delta_{2},\delta_{3},\delta_4,\delta_5)$. Let $G(x)=e^{-\frac{1}{x}}$ and $H(x)=1-e^{-x},$ for $x>0.$ Also, let $(X_{1},X_{2},X_{3})$ be selected with probability $P(N_{1}=3)=p(3)=1/5,$ $(X_{1},X_{2},X_{3},X_{4})$ be selected with probability $P(N_{1}=4)=p(4)=2/5$ and $(X_{1},X_{2},X_{3},X_{4},X_{5})$ be selected with probability $P(N_{1}=5)=p(5)=2/5;$ further let $(Y_{1},Y_{2},Y_{3})$ be selected with probability $P(N_{2}=3)=p(3)=1/5,$ $(Y_{1},Y_{2},Y_{3}, Y_{4})$ be selected with probability $P(N_{2}=4)=p(4)=3/5$ and $(Y_{1},Y_{2},Y_{3}, Y_{4},Y_{5})$ be selected with probability $P(N_{2}=5)=p(5)=1/5.$ Here, $N_{1}\leq_{st}N_{2}.$ Consider the Gumbel copula { with} generators $\psi_1(x) = e^{-x^\frac{1}{a}}$ and $\psi_2(x) = e^{-x^\frac{1}{b}},~x>0,$ with $a=1.4$ and $b=2.5$. Here, all the conditions of Theorem \ref{th3} are satisfied except $G\geq H.$ Figure $1(b)$ displays the difference between  $\bar{F}_{X_{1:{N_{1}}}}(x)$ and $\bar{F}_{Y_{1:{N_{2}}}}(x)$ crosses $x$ axis for some $x\geq 0,$ which violates the statement of Theorem \ref{th3}. 
\end{counterexample}

In Theorem \ref{th1}, we have established the conditions for which ``$\boldsymbol{\alpha}\succeq^{w}\boldsymbol{\beta}\Rightarrow X_{{N_1}:{N_1}}\geq_{st} Y_{{N_2}:{N_2}}$'' holds. Now, a natural question that arises is whether the result is true for random minima. The following theorem provides an affirmative answer to this question.
\begin{theorem}\label{th4}
 Let Assumption \ref{ass1} hold with $\boldsymbol{\gamma}=\boldsymbol{\delta},~\boldsymbol{\alpha},~\boldsymbol{\beta},~\boldsymbol{\delta}\in\mathcal{E}_+(\mathcal{D}_+)$ and $G\geq H.$ Also, suppose $\phi_{2}\circ\psi_{1}$ is super-additive and $\psi_1$ or $\psi_2$ is log-convex. Then, $\boldsymbol{\alpha}\succeq^{w}\boldsymbol{\beta}\Rightarrow X_{1:{N_{1}}}\leq_{st} Y_{1:{N_{2}}},$ whenever $N_{1}\leq_{st}N_{2}$.
\end{theorem}
\begin{proof}
The proof is quite similar to that of Theorem \ref{th3}, and is therefore omitted for the sake of brevity.
\end{proof}
{The following counterexample shows that the log-convexity condition is required for Theorem \ref{th4}.
\begin{counterexample}\label{cex1.2.}
Suppose $X_i \sim \text{Kw-G}(x,\alpha_i,\gamma_i;G)$ and $Y_i \sim \text{Kw-G}(x,\alpha_{i},\delta_i;H),$ for $i=1,2,3,4,5$. Further, let $(\alpha_1,\alpha_2 ,\alpha_3,\alpha_4,\alpha_5)=(7.1,2.9,1.56,0.03,0.201)$, $(\beta_{1},\beta_{2},\beta_{3},\beta_{4},\beta_{5})=(8.1,3.009,$
$2.6,1.06,0.01)$ and $(\gamma_1, \gamma_2,\gamma_3,\gamma_4,\gamma_5)=(\delta_1, \delta_2, \delta_3,\delta_4,\delta_5)=(5.02,2.05,1.09,0.01,0.001)$. It is easy to see that $(\alpha_1,\alpha_2,\alpha_3,\alpha_4,\alpha_5)$
$\succeq^{w}(\beta_{1},\beta_{2},\beta_{3},\beta_4,\beta_5)$. Let $H(x)=e^{-\frac{1}{x}}$ and $G(x)=1-e^{-x},$ for $x>0.$ Also, let $(X_{1},X_{2},X_{3})$ be selected with probability $P(N_{1}=3)=p(3)=1/5,$ $(X_{1},X_{2},X_{3},X_{4})$ be selected with probability $P(N_{1}=4)=p(4)=2/5$ and $(X_{1},X_{2},X_{3},X_{4},X_{5})$ be selected with probability $P(N_{1}=5)=p(5)=2/5;$ further, let $(Y_{1},Y_{2},Y_{3})$ be selected with probability $P(N_{2}=3)=p(3)=1/5,$ $(Y_{1},Y_{2},Y_{3}, Y_{4})$ be selected with probability $P(N_{2}=4)=p(4)=3/5$ and $(Y_{1},Y_{2},Y_{3}, Y_{4},Y_{5})$ be selected with probability $P(N_{2}=5)=p(5)=1/5,$ which satisfy the condition $N_{1}\leq_{st}N_{2}.$ Choose the Gumbel-Hougaard copula { with} generators $\psi_1 (x)= e^{1-(1+x)^\frac{1}{a}}$ and $\psi_2(x) = e^{1-(1+x)^\frac{1}{b}},~x>0,$ with $a=4.0002$ and $b=4.0001$. So, all the conditions of Theorem \ref{th4} are satisfied except the log-convexity of $\psi_1$ and $\psi_2$. Figure \ref{c3..} displays the difference between $\bar{F}_{X_{1:{N_1}}}(x)$ and $\bar{F}_{Y_{1:{N_2}}}(x)$ crosses $x$ axis for some $x\geq 0.$ Therefore, Theorem \ref{th4} can not be possible without the log-convexity property of the generators.
	\begin{figure}[h!]
		\begin{center}
			\includegraphics[height=2.0in]{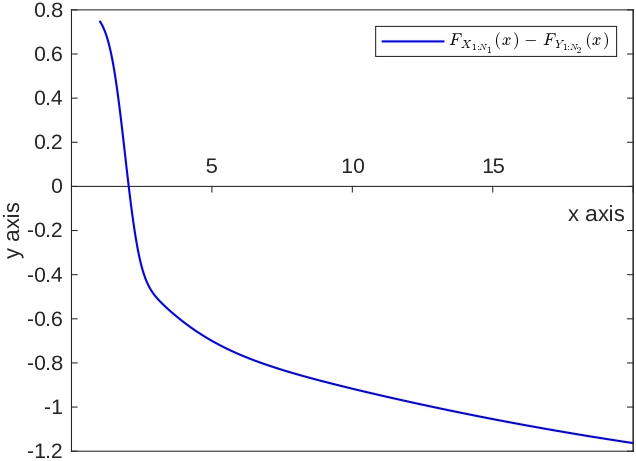}
			\caption{Plot of $\bar{F}_{X_{1:{N_{1}}}}(x)-\bar{F}_{Y_{1:{N_2}}}(x)$ as in Counterexample \ref{cex1.2.} when $\psi_1$ and $\psi_2$ are both log-concave.  }
			\label{c3..}
		\end{center}
	\end{figure}
  \end{counterexample}}
  \begin{remark}
      It is also of interest to show that if we consider $\boldsymbol{\gamma}=\boldsymbol{\delta}(= \boldsymbol{1}_{n})$ in Theorem \ref{th4}, then under the condition $\frac{1-\psi_1}{\psi^{`}_{1}}$ or $\frac{1-\psi_2}{\psi^{`}_{2}}$ (instead of $\psi_1$ or $\psi_2$) being log-convex, the result can also be true.
  \end{remark}
{\begin{remark}
	If we take Independence copula in place of Archimedean survival copula, then Theorem \ref{th4} and Theorem \ref{th3} are extensions of Theorem $4.1$ and Theorem $4.2(i)$ of \cite{kundu2018w}, respectively, when the number of observations are not random and same.
\end{remark}}
 It is important to note that the condition $\phi_2\circ\psi_1$ is supper-additive is nothing but to say that they are related with positively ordered copula families when generators $\psi_1$ and $\psi_2$ belong to same family of copulas. For example, we can consider Gumbel-Hougaard, Clayton, Ali-Mikhail-Haq (AMH) and Frank copulas (see \cite{nelsen2006introduction}). Moreover, the log-convexity of a generator of an Archimedean copula leads to the left tail decreasing in sequence property (see \cite{you2014optimal}). There are Archimedean copulas with generators satisfying the log-convexity  and log-concavity properties. For example, Clayton copula, independence copula, AMH copula (for non-negative valued parameters) and Gumbel copula all have log-convex generators, while AMH copula with negative valued parameters has log-concave generator. {For various subfamilies of Archimedean survival copulas, an interpretation of the superadditive
property of $\phi_2\circ\psi_1$ is that Kendall’s $\tau$ of the $2$-dimensional copula with generator $\psi_2$ is larger
than that of the copula with generator $\psi_1$ and consequently, is more positive dependent. Similar observation can be found for other measures like Blomqvist’s $\beta$, Spearman’s rho and Gini’s $\gamma$ (see \cite{nelsen2006introduction}).}

Next, we develop a necessary and sufficient condition for comparing random minima under same (random) sample size in the sense of the hazard rate order when the observations are independent and heterogeneous.  {It is important to note that in the remaining part of this section, to establish the ordering results, we will consider $N_{1}\overset{st}{=}N_{2}\overset{st}{=}N,$ and $\psi_{1}(x)=\psi_{2}(x)=\psi(x)=e^{-x},~x>0$. The following theorems will be used to prove some results in this section.
\begin{theorem}(Theorem $3.1$ of \cite{chowdhury2024})\label{ch24_3.1}
    Suppose $X_{n:n}\sim \bar{F}_{n:n}(x)$ and $X_{n:n}\sim \bar{G}_{n:n}(x).$ Also, let the support of an
integer-valued random variable $N$ having pmf $p(n)$ be $N_{+}$. Now, for all $x\geq l$, if $\frac{\bar{F}_{1:n}(x)}{\bar{G}_{1:n}(x)}$ is increasing
in $n\in N_{+}$, then $X_{1:n}\leq_{hr} Y_{1:n}$ implies $X_{1:N}\leq_{hr} Y_{1:N}$.
\end{theorem}
\begin{theorem}(Theorem $3.2$ of \cite{chowdhury2024})\label{ch24_3.2}
    Suppose $X_{n:n}\sim \bar{F}_{n:n}(x)$ and $X_{n:n}\sim \bar{G}_{n:n}(x).$ Also, let the support of an
integer-valued random variable $N$ having pmf $p(n)$ be $N_{+}$. Now, for all $x\geq u$, if $\frac{\bar{F}_{1:n}(x)}{\bar{G}_{1:n}(x)}$ is decreasing
in $n\in N_{+}$, then $X_{1:n}\geq_{hr} Y_{1:n}$ implies $X_{1:N}\geq_{hr} Y_{1:N}$.
\end{theorem}

\begin{theorem}\label{cor7}
	Let Assumption \ref{ass1} hold with $\boldsymbol{\alpha}=\boldsymbol{\beta}=\boldsymbol{1}_{n},$ $r_{g}\leq r_{h}.$ Then, $\sum_{i=1}^{n}\gamma_{i}\leq \sum_{i=1}^{n}\delta_{i}$ holds if and only if $ X_{1:N}\geq_{hr} Y_{1:N}.$
\end{theorem}}
\begin{proof}
{According to Theorem \ref{ch24_3.2}, we only need to show that the following two conditions hold:
\begin{itemize}
    \item [(i)] $\frac{\bar{F}_{X_{1:{n}}}(x)}{\bar{F}_{Y_{1:{n}}}(x)}$ is decreasing in $n;$
    \item [(ii)] $X_{1:n}\geq_{hr} Y_{1:n}\Rightarrow r_{X_{1:n}}(x)\leq r_{Y_{1:n}}(x).$
\end{itemize}
Let
\begin{equation}
A(n)=\frac{\bar{F}_{X_{1:{n}}}(x)}{\bar{F}_{Y_{1:{n}}}(x)}=\frac{(1-G(x))^{\sum_{i=1}^{n}\gamma_{i}}}{(1-H(x))^{\sum_{i=1}^{n}\delta_{i}}}={\bar{G}(x)}^{\sum_{i=1}^{n}\gamma_{i}-\sum_{i=1}^{n}\delta_{i}}\times \left(\frac{\bar{G}(x)}{\bar{H}(x)}\right)^{\sum_{i=1}^{n}\delta_{i}}.
\end{equation}
To complete the first part we only need to show that  ${\bar{G}(x)}^{\sum_{i=1}^{n}\gamma_{i}-\sum_{i=1}^{n}\delta_{i}}$ is decreasing in $n,$  which can be proved by the given condition $\sum_{i=1}^{n}\gamma_{i}\leq \sum_{i=1}^{n}\delta_{i}.$ For the second part, we need to establish
	\begin{align*}
		r_{X_{1:n}}(x)\leq r_{Y_{1:n}}(x)
		\iff r_{g}(x)\sum\limits_{i=1}^{n}\gamma_{i}\leq r_{h}(x) \sum\limits_{i=1}^{n}\delta_{i}.
	\end{align*}
	From the given assumptions, we can then prove the above inequality, which will establish the theorem.} 
\end{proof}
{The next theorem states that if ${\boldsymbol\gamma}\succeq^{m}{\boldsymbol\delta},$ then we can also compare the random minima according to the hazard rate order, under the conditions $\boldsymbol{\alpha}=\boldsymbol{\beta}=\alpha\bm{1}_{n}$ and $G=H$. 
\begin{theorem}\label{th6}
	Let Assumption \ref{ass1} hold with $\boldsymbol{\alpha}=\boldsymbol{\beta}=\alpha\bm{1}_{n}.$ Then,
	${\boldsymbol\gamma}\succeq^{m}{\boldsymbol\delta}\Rightarrow X_{1:{N}}\geq_{hr}Y_{1:{N}},$ whenever $G=H$.
\end{theorem}
\begin{proof}
	Using Theorem \ref{ch24_3.2}, to obtain the required result, we only need to show that the following two conditions hold:
		\begin{itemize}
			\item [(i)] $\frac{\bar{F}_{X_{1:{n}}}(x)}{\bar{F}_{Y_{1:{n}}}(x)}$ is decreasing in $n;$
			\item [(ii)] $X_{1:n}\geq_{hr} Y_{1:n}\Rightarrow r_{X_{1:n}}(x)\leq r_{Y_{1:n}}(x).$
		\end{itemize}
	Utilizing a similar argument as in Theorem \ref{cor7}, one can easily prove the first part. As the survival function of $\text{Kw-G}$ distribution is same as distribution function of proportional hazard rate model if we take $G^\gamma(x),~\gamma>0,$ as the baseline. Thus, we can prove that $r_{X_{1:n}}(x)\leq r_{Y_{1:n}}(x)$ under the condition ${\boldsymbol\gamma}\succeq^{m}{\boldsymbol\delta}$ following the same lines as in Theorem $3.1$ of \cite{li2019hazard}, which completes the second part. Hence, the theorem.
\end{proof}

\begin{remark}
It is important to mention that Theorem \ref{th6} is an extension of Theorem $4.4$ of \cite{kundu2018w}, for the case of random (same) sample size.
\end{remark}}
 {The following theorem provides additional sufficient conditions on $\bm{\alpha}$ and $\bm{\beta}$ for comparing random minima according to the hazard rate order.}
    \begin{theorem}\label{hr_1}
        Let Assumption \ref{ass1} hold with $\boldsymbol{\gamma}=\boldsymbol{\delta},$ $\bm{\alpha},~\bm{\beta},~\boldsymbol{\gamma}\in \mathcal{E}_{+}(\mathcal{D}_{+})$. Then, $\boldsymbol{\alpha}\succeq^{w}\boldsymbol{\beta}\Rightarrow X_{1:N}\geq_{hr} Y_{1:N},$ whenever $G=H$.
    \end{theorem}
    \begin{proof}
        { Employing Theorem \ref{ch24_3.2}, to obtain the required result, we only need to verify the following two conditions:
        \begin{itemize}
    \item [(i)] $\frac{\bar{F}_{X_{1:{n}}}(x)}{\bar{F}_{Y_{1:{n}}}(x)}$ is decreasing in $n;$
    \item [(ii)] $ X_{1:n}\geq_{hr} Y_{1:n} \Rightarrow r_{X_{1:n}}(x)\leq r_{Y_{1:n}}(x).$
\end{itemize}
Denote $B(n)=\frac{\bar{F}_{X_{1:{n}}}(x)}{\bar{F}_{Y_{1:{n}}}(x)}=\frac{\prod_{i=1}^{n}(1-G^{\alpha_{i}})^\gamma}{\prod_{i=1}^{n}(1-G^{\beta_{i}})^{\gamma}}.$ To check  $\frac{\bar{F}_{X_{1:{n}}}(x)}{\bar{F}_{Y_{1:{n}}}(x)}$ is decreasing in $n,$ it is sufficient to show that $\frac{B(n+1)}{B(n)}>1,$ which can be concluded from the given assumption. Now, applying Theorem $3.3$ of \cite{kundu2018w}, we can conclude that $r_{X_{1:n}}(x)\leq r_{Y_{1:n}}(x)$ under the condition $\boldsymbol{\alpha}\succeq^{w}\boldsymbol{\beta},$ which completes the second part. Hence, the theorem.}
       \end{proof}

    In the above theorem, we have considered the case when the baseline distributions are equal. But, it will be of more interest to study the case when the baseline distributions are not equal. The following result provides a positive answer to this question.
\begin{theorem}\label{th12}
	Let Assumption \ref{ass1} hold with $\boldsymbol{\gamma}=\boldsymbol{\delta},$ $\bm{\alpha},~\bm{\beta},~\boldsymbol{\gamma}\in \mathcal{E}_{+}(\mathcal{D}_{+}).$ Then, $\boldsymbol{\alpha}\succeq^{w}\boldsymbol{\beta}\Rightarrow X_{1:N}\leq_{hr} Y_{1:N},$ whenever $r_{g}\geq r_{h}$.
\end{theorem}
\begin{proof}
First, we consider the case when $\bm{\alpha},~\bm{\beta}\geq\bm{1}_n.$ The other part can be done in a similar way. {Applying Theorem \ref{ch24_3.1}, we first need to prove that $\frac{\bar{F}_{X_{1:{n}}}(x)}{\bar{F}_{Y_{1:{n}}}(x)}$ is increasing in $n$ and then $X_{1:n}\leq_{hr} Y_{1:n}.$ The first part can be established using the same steps as in Theorem \ref {hr_1}. To complete the second part, we have to check whether $r_{X_{1:n}}(x)\geq r_{Y_{1:n}}(x)$, where  
	\begin{equation}\label{r1}	    
	r_{X_{1:n}}(x)=\sum\limits_{i=1}^{n} \frac{\alpha_{i}\gamma_{i}g(x) G^{\alpha_{i}-1}(x)}{1-G^{\alpha_{i}}(x)}
	=\sum\limits_{i=1}^{n} \frac{\alpha_{i}\gamma_{i}r_{g}(x)(1-G(x)) G^{\alpha_{i}-1}(x)}{1-G^{\alpha_{i}}(x)}
	\end{equation}
    represents the hazard rate function of $X_{1:n}.$ Similarly, we can write an expression for $r_{Y_{1:n}}(x)$ from \eqref{r1} upon replacing $\alpha_{i}$ by $\beta_{i},$ for $i=1,\ldots,n.$
According to Lemma $3.3$ of \cite{kundu2016some} and Lemma $7$ of \cite{balakrishnan2015stochastic}, we can then show that $r_{X_{1:n}}(x)\geq r_{Y_{1:n}}(x)$ under the condition $\boldsymbol{\alpha}\succeq^{w}\boldsymbol{\beta}.$ This establishes the theorem.}
\end{proof} 

In the following result, we establish some sufficient conditions for comparing two random maxima according to the reversed hazard rate order. { To prove the next Theorem, we will use the following result.

\begin{theorem}(Theorem $3.4$ of \cite{chowdhury2024})\label{ch24_3.3}
    Suppose $X_{n:n}\sim \bar{F}_{n:n}(x)$ and $X_{n:n}\sim \bar{G}_{n:n}(x).$ Also let the support of a positive
integer valued random variable $N$ having pmf $p(n)$ be $N_{+}$. Now, for all $x\geq l$, if $\tilde{r}_{Y_{n:n}}(x)$,
the reversed hazard rate function of $Y_{n:n}$ is increasing in $n\in N_{+}$ and $\frac{F_{n:n}(x)}{G_{n:n}(x)}$ is increasing
in $n\in N_{+}$, then $X_{n:n}\geq_{rh} Y_{n:n}$ implies $X_{N:N}\geq_{rh} Y_{N:N}$.
\end{theorem}}

\begin{theorem}\label{th13}
	Let Assumption \ref{ass1} hold with $\boldsymbol{\gamma}=\boldsymbol{\delta}=\gamma\bm{1}_n ,$ $\bm{\alpha},~\bm{\beta}\in \mathcal{E}_{+}(\mathcal{D}_{+})$. Then, $\boldsymbol{\alpha}\succeq^{w}\boldsymbol{\beta}\Rightarrow X_{N:N}\geq_{rh} Y_{N:N},$ whenever $\bm{\alpha},~\bm{\beta}\geq\bm{1}_n,$ $G\leq H$ and $g\geq h$.
\end{theorem}
\begin{proof}
{The reversed hazard rate function of $X_{n:n}$ can be written as
	\begin{equation*}	    
	\tilde{r}_{X_{n:n}}(x)=\sum\limits_{i=1}^{n} A(\alpha_i)B(\alpha_i),
	\end{equation*}
	where $A(\alpha_i)=\frac{\alpha_{i}g(x) G^{\alpha_{i}-1}(x)}{1-G^{\alpha_{i}}(x)}$ and $B(\alpha_i) =\frac{\gamma(1-G^{\alpha_{i}}(x))^{\gamma} }{1-(1-G^{\alpha_{i}}(x))^{\gamma} },$ for $i=1,\ldots,n.$
    Using the same arguments as in Theorem \ref{th12}, we can easily verify that $\frac{{F}_{X_{n:{n}}}(x)}{{F}_{Y_{n:{n}}}(x)}$ is increasing in $n$ using the given assumptions. Now,  
\begin{equation}
\tilde{r}_{X_{{n+1}:{n+1}}}(x)-\tilde{r}_{X_{n:n}}(x)=A(\alpha_{n+1})B(\alpha_{n+1})\geq 0, 
\end{equation}
which shows that $\tilde{r}_{X_{n:n}}(x)$ is increasing in $n.$ Finally, by applying Theorem \ref{ch24_3.3}, we only need to prove that $X_{n:n}\geq_{rh} Y_{n:n}$ which is equivalent to establishing that $\tilde r_{X_{n:n}}(x)\geq \tilde{r}_{Y_{n:n}}(x).$ } From Theorem \ref{th12}, $A(\alpha_i)$ is decreasing and convex in $\alpha_i$. Similarly, after some mathematical calculations, $B(\alpha_i)$ is also {seen to be} decreasing and convex in $\alpha_i$. We also know  that every function that is convex and symmetric under permutations of the arguments is also Schur-convex. Therefore, $\tilde{r}_{X_{n:n}}(x)$ is decreasing and Schur-convex in $\bm{\alpha}\in\mathcal{E}_{+}(\mathcal{D}_{+}).$ {Finally, employing Theorem $A.8$ of \cite{Marshall2011}, we obtain $\tilde r_{X_{n:n}}(x)\geq \tilde{r}_{Y_{n:n}}(x)$ under the condition $\boldsymbol{\alpha}\succeq^{w}\boldsymbol{\beta}$. Hence, the theorem.}
\end{proof} 
Sometimes, we need to compare two probability distributions by their variability or spread instead of doing it based on survival functions or hazard rate functions. In this regard, dispersive ordering is one of the basic concepts for comparing based on spread. In the following, we develop some sufficient conditions for comparing two random minima under same (random) sample size in terms of dispersive order. We use the idea of Theorem $3.B.20$ of \cite{shaked2007stochastic} and Theorem \ref{th12} to establish these results.
\begin{theorem}\label{th13dis1}
	Let Assumption \ref{ass1} hold with $\boldsymbol{\gamma}=\boldsymbol{\delta},$ $\bm{\alpha},~\bm{\beta},~\boldsymbol{\gamma}\in \mathcal{E}_{+}(\mathcal{D}_{+}).$ Then, $\boldsymbol{\alpha}\succeq^{w}\boldsymbol{\beta}\Rightarrow X_{1:N}\leq_{disp} Y_{1:N},$ whenever $\bm{\alpha},~\bm{\beta}\leq\bm{1}_n,$ $r_{g}\geq r_{h}$ and $r_{g}(x)$ or $r_{h}(x)$ is decreasing.
\end{theorem}
Using the same argument as in Theorem \ref{th13dis1}, we can establish the following theorem. 
\begin{theorem}\label{th16}
	Under the assumptions of Theorem \ref{th6}, we have
	${\boldsymbol\gamma}\succeq^{m}{\boldsymbol\delta}\Rightarrow X_{1:N}\leq_{disp}Y_{1:N},$ whenever $r_{g}(x)$ is decreasing.
\end{theorem}

We complete this Section with the following theorems related to likelihood ratio order. 
\begin{theorem} \label{lr1}
	Suppose Assumption \ref{ass1} hold with $\bm{\gamma}=\bm{\delta}=\gamma\bm{1}_{n}.$ Also, for $i=1,\cdots,p,$ let $\alpha_{i}=\alpha_1,$ $\beta_i=\beta_1$ and for $i={p+1},\cdots,n,$ $\alpha_{i}=\alpha_2,$ $\beta_i=\beta_2$ satisfying {$\alpha_{1}\leq\alpha_{2}\leq \beta_{2}\leq \beta_{1}.$} Then, $(\alpha_{1}\bm{1}_{p},\alpha_{2}\bm{1}{q})\succeq^{m}(\beta_{1}\bm{1}_{p},\beta_{2}\bm{1}{q})\Rightarrow X_{1:N}\geq_{lr}Y_{1:N}. $ 
\end{theorem}
\begin{proof}
    {Under the given assumption, the density functions of $X_{1:n}$ and $Y_{1:n}$ are, respectively, given by 
    $f_{X_{1:n}}(x)=\sum_{i=1}^{n}\frac{\gamma\alpha_{i}r_{g}(x)G^{\alpha_{i}}(x)}{1-G^{\alpha_{i}}(x)}\overset{sign}{=}\frac{p\alpha_{1}G^{\alpha_{1}}(x)}{1-G^{\alpha_{1}}(x)}+\frac{q\alpha_{2}G^{\alpha_{2}}(x)}{1-G^{\alpha_{2}}(x)}$ and $f_{Y_{1:n}}(x)=\sum_{i=1}^{n}\frac{\gamma\beta_{i}r_{g}(x)G^{\beta_{i}}(x)}{1-G^{\beta_{i}}(x)}\overset{sign}{=}\frac{p\beta_{1}G^{\beta_{1}}(x)}{1-G^{\beta_{1}}(x)}+\frac{q\beta_{2}G^{\beta_{2}}(x)}{1-G^{\beta_{2}}(x)},$ where $n=p+q.$ Applying Theorem $3.5$ of \cite{chowdhury2024}, we need to establish the following three conditions:
    \begin{itemize}
        \item[(i)] $\frac{f_{X_{1:n}}(x)}{f_{Y_{1:n}}(x)}$ is increasing in $n;$
        \item[(ii)] $n_{1}\leq n_{2}$ implies $X_{1:{n_1}}\geq_{lr} X_{1:n_{2}};$
        \item[(iii)] $ X_{1:{n}}\geq_{lr} Y_{1:n}.$
    \end{itemize}
    To prove Part $(i)$, we have to show that
    $A(n+1)=\frac{f_{X_{1:{n+1}}}(x)}{f_{Y_{1:{n+1}}}(x)}\geq \frac{f_{X_{1:n}}(x)}{f_{Y_{1:n}}(x)}=A(n),$ which is equivalent to showing that $\frac{p\alpha_{1}G^{\alpha_{1}}(x)}{1-G^{\alpha_{1}}(x)}+\frac{q\alpha_{2}G^{\alpha_{2}}(x)}{1-G^{\alpha_{2}}(x)}\geq \frac{p\beta_{1}G^{\beta_{1}}(x)}{1-G^{\beta_{1}}(x)}+\frac{(q+1)\beta_{2}G^{\beta_{2}}(x)}{1-G^{\beta_{2}}(x)}.$ Using the given condition $\alpha_{1}\leq\alpha_{2}\leq \beta_{2}\leq \beta_{1}$, we can easily see that the above inequality holds. We can complete Part $(ii)$ by using the second part of Theorem $1.C.31$ of \cite{shaked2007stochastic} and the given assumption $\alpha_{1}\leq\alpha_{2}.$ Finally, by following the steps of the proof of Theorem $5.3 $ of \cite{Das2021peis}, we can prove Part $(iii)$ under the given condition $(\alpha_{1}\bm{1}_{p},\alpha_{2}\bm{1}{q})\succeq^{m}(\beta_{1}\bm{1}_{p},\beta_{2}\bm{1}{q})$. Hence, the theorem.}
\end{proof}
 {Next, we develop another new sufficient conditions on the parameters such that the likelihood ratio order exists between two random minima $X_{1:{N}}$ and $Y_{1:{N}}$. To prove the results, we need the following lemma and theorem.}
\begin{lemma}\label{lem2}
	Suppose $d_1$, $d$ and $d_2$ are three nonnegative functions. Then, $\frac{d_1(x) +d(x)}{d_2(x) +d(x)}$ is decreasing in $x$ if $\frac{d_1}{d_2},$ $\frac{d}{d_2}$ and $\frac{d_1}{d}$ are decreasing in $x$. 
\end{lemma}
\begin{theorem}\label{cor_lrth1}
	Suppose Assumption \ref{ass1} hold with $\psi_{1}=\psi_{2}=\psi,$ $F=G$ and $\bm{\alpha}=\bm{\beta}=\boldsymbol{1}_n.$ Also, for $i=1,\cdots,p,$ let $\gamma_{i}=\gamma_1,$ $\delta_i=\delta_1$ and $i={p+1},\cdots,n,$ $\gamma_{i}=\gamma,$ $\delta_i=\gamma.$
	If $\frac{\psi''}{\psi'}$ is decreasing, $\frac{\psi\ln\psi}{\psi'}$ is increasing and concave, then $\gamma_1\geq \gamma\geq \delta_1\Rightarrow X_{1:n}\leq_{lr}Y_{1:n}. $ Here, $p\geq 1$ and $p+q=n\geq 2.$
\end{theorem}

\begin{proof}
	To obtain the required result, we need to establish
	$$\frac{{f}_{X_{1:n}}(p,q)}{{f}_{Y_{1:n}}(p,q)}=\kappa_1(x)\kappa_2(x)$$ is decreasing in $x,$ where
	
	$${f}_{X_{1:n}}(p,q)=\psi'(p\phi(g^{\gamma_1}(x))+q\phi(g^{\gamma}(x)))\times[p\gamma_1g'(x)g^{\gamma_1 -1}(x)\phi'(g^{\gamma_1}(x))+q\gamma g'(x)g^{\gamma -1}(x)\phi'(g^{\gamma}(x))],$$
	$${f}_{Y_{1:n}}(p,q)=\psi'(p\phi(g^{\delta_1}(x))+q\phi(g^{\gamma}(x)))\times[p\delta_1g'(x)g^{\delta_1 -1}(x)\phi'(g^{\delta_1}(x))+q\gamma g'(x)g^{\gamma -1}(x)\phi'(g^{\gamma}(x))],$$
	\begin{equation*}\label{lr10}
		\kappa_1(x)=\frac{\psi'(p\phi(g^{\gamma_1}(x))+q\phi(g^{\gamma}(x)))}{\psi'(p\phi(g^{\delta_1}(x))+q\phi(g^{\gamma}(x)))}\text{ and } \kappa_2(x)=\frac{[p\gamma_1g'(x)g^{\gamma_1 -1}(x)\phi'(g^{\gamma_1}(x))+q\gamma g'(x)g^{\gamma -1}(x)\phi'(g^{\gamma}(x))]}{[p\delta_1g'(x)g^{\delta_1 -1}(x)\phi'(g^{\delta_1}(x))+q\gamma g'(x)g^{\gamma -1}(x)\phi'(g^{\gamma}(x))]}.
	\end{equation*}
	
	As $\kappa_1(x)$ and $\kappa_2(x)$ are both positive-valued,  we only need to verify that the functions are decreasing in $x.$ Therefore, to prove $\kappa'_1(x)\leq 0,$ let us suppose $s_1=p\phi(g^{\gamma_1}(x))+q\phi(g^{\gamma}(x))$ and $s_2=p\phi(g^{\delta_1}(x))+q\phi(g^{\gamma}(x)).$
	Then,
	\begin{align}
		\kappa'_1{(x)}&\overset{sign}{=} s^{'}_{1}\frac{\psi^{''}(s_1)}{\psi^{'}(s_1)}-s^{'}_{2}\frac{\psi^{''}(s_2)}{\psi^{'}(s_2)}\nonumber\\
		&=\left[p\left[\frac{\psi(t)\ln\psi(t)}{\psi'(t)}\right]_{t=\phi(g^{\gamma_1}(x))}+q\left[\frac{\psi(t)\ln\psi(t)}{\psi'(t)}\right]_{t=\phi(g^{\gamma}(x))}\right]\times \frac{\psi^{''}(s_1)}{\psi^{'}(s_1)}\nonumber\\
		&-\left[p\left[\frac{\psi(t)\ln\psi(t)}{\psi'(t)}\right]_{t=\phi(g^{\delta_1}(x))}+q\left[\frac{\psi(t)\ln\psi(t)}{\psi'(t)}\right]_{t=\phi(g^{\gamma}(x))}\right]\times \frac{\psi^{''}(s_2)}{\psi^{'}(s_2)}.		
	\end{align}
	Now, $\gamma_1\geq \gamma\geq \delta_1$ implies $s_1\geq s_2$ and $\phi(g^{\gamma_{1}}(x))\geq \phi(g^{\delta_1}).$
	By using the given assumptions, we can then easily check that $\kappa_1(x)$ is decreasing in $x.$ Again, $\kappa_{2}(x)$ can be written in the form
	\begin{align}
		\kappa_{2}(x)\overset{sign}{=}\frac{\frac{p\psi(u)\ln\psi(u)}{\psi'(u)}|_{u=\phi(g^{\gamma_1}(x))}+\frac{q\psi(u)\ln\psi(u)}{\psi'(u)}|_{u=\phi(g^{\gamma}(x))}}{\frac{p\psi(u)\ln\psi(u)}{\psi'(u)}|_{u=\phi(g^{\delta_1}(x))}+\frac{q\psi(u)\ln\psi(u)}{\psi'(u)}|_{u=\phi(g^{\gamma}(x))}}.
	\end{align}
	Denote $h_1(x)=\frac{\frac{\psi(u)\ln\psi(u)}{\psi'(u)}|_{u=\phi(g^{\gamma_1})}}{\frac{\psi(u)\ln\psi(u)}{\psi'(u)}|_{u=\phi(g^{\delta_1})}}$, $h_2(x)=\frac{\frac{\psi(u)\ln\psi(u)}{\psi'(u)}|_{u=\phi(g^{\gamma_1})}}{\frac{\psi(u)\ln\psi(u)}{\psi'(u)}|_{u=\phi(g^{\gamma})}}$ and $h_3(x)=\frac{\frac{\psi(u)\ln\psi(u)}{\psi'(u)}|_{u=\phi(g^{\gamma})}}{\frac{\psi(u)\ln\psi(u)}{\psi'(u)}|_{u=\phi(g^{\delta_1})}}.$
	Now, using the log-concavity property of $\frac{\psi\ln\psi}{\psi'},$ it is easy to observe that $h_{1}^{'}(x)\leq 0.$ Similarly, we can show that both $h_{2}(x)$ and $h_{3}(x)$  are decreasing in $x$. Finally Lemma \ref{lem2} can be used to complete the proof of the theorem.
\end{proof}

{\begin{theorem}\label{lrth1}
	Suppose Assumption \ref{ass1} hold with $F=G$ and $\bm{\alpha}=\bm{\beta}=\boldsymbol{1}_n.$ Also, for $i=1,\cdots,p,$ let $\gamma_{i}=\gamma_1,$ $\delta_i=\delta_1$ and $i={p+1},\cdots,n,$ $\gamma_{i}=\gamma,$ $\delta_i=\gamma.$ Then $\gamma_1\geq \gamma\geq \delta_1\Rightarrow X_{1:N}\leq_{lr}Y_{1:N}. $ Here, $p\geq 1$ and $p+q=n\geq 2.$
\end{theorem}
\begin{proof}
	By applying Theorem \ref{lr1}, we can easily obtain the required result using the same idea as in Theorem \ref{cor_lrth1}. Hence, the details are omitted.
\end{proof}

The proof of the following theorem is similar to the last theorem, and so it is not presented.
\begin{theorem}\label{lrth2}
	Suppose Assumption \ref{ass1} {holds} with $\bm\gamma=\bm\delta=\bm{1}_n.$ Also, for $i=1,\cdots,p,$ let $\alpha_{i}=\alpha_1,$ $\beta_i=\beta_1$ and $i={p+1},\cdots,n,$ $\alpha_{i}=\alpha,$ $\beta_i=\alpha.$ Then, $\alpha_1\geq \alpha\geq \beta_1\Rightarrow X_{1:N}\leq_{lr}Y_{1:N}. $ Here, $p\geq 1$ and $p+q=n\geq 2.$
\end{theorem}}
Note that the sufficient conditions provided in the above theorem is distribution-free and so the presented result is applicable to all life distributions.
{\begin{remark} 
    It should be noted that the results for hazard rate, reversed hazard rate and likelihood ratio orders are based on independent observations; these results remain open problems in the case of dependent observations.
\end{remark}}
{
\section{Applications}\label{app}
 Ordering results between two random extremes is useful in different problems. In the following, we describe some applications of our results in biostatistics and transportation theory for illustrational purpose.

\subsection{Biostatistics} 
    Consider the case of a cancer patient treatment. Let $N$ be the number of carcinogenic cells that are left active after the first shot of the treatment is given. Let $X_i$, $i=1,\ldots,N,$ be the incubation time for the $i$th clonogenic cell. Suppose the expiry of $1$-out-of-$N$ and $N$-out-of-$N,$ $N\geq1,$ latent factors of the system need to be activated. Then, the failure time for the patient is $X_{1:N}$ and $X_{N:N},$ respectively. These series and parallel system cases are called first and last activation schemes (see \cite{Cooner2007}). Here, we concentrate on the failure times of the first and last activation schemes.

Under the above set up, we assume that the failure times of the cells are interdependently distributed as $\text{Kw-W}$ with CDF of the $i$th cell being $G(x)=1-[1-\{1-exp^{-(\eta x)^{\lambda}}\}^{\alpha_i}]^{\gamma_i},$ $x,\alpha_i, \gamma_i, \lambda, \eta>0,$ for $i=1,\ldots,n.$ Here, we consider Weibull distribution as the parent CDF of the $\text{Kw-G}$ distribution.  Now, let us assume two cancer patients who are under a treatment. Suppose $\{X_1,\ldots,X_n\}$ are the incubation times for the $n$ carcinogenic cells and $N_1$ is the number of carcinogenic cells that are left active for the first patient after first chemotherapy, where $X_i\sim\text{Kw- W}(x,\alpha_{i},\gamma_{i}, \eta, \lambda)$ and $N_1\sim Poission(\lambda_1)$. Similarly, for the second patient, let $\{Y_1,\ldots,Y_n\}$ be the incubation times for the $n$ carcinogenic cells and $N_2$ be the number of carcinogenic cells that are left active, where $Y_i\sim\text{Kw- W}(x,\beta_{i},\delta_{i}, \eta,\zeta)$ and $N_2\sim Poission(\lambda_2)$. Then, under the given setup, by using Theorem \ref{th1} (Theorem \ref{th4}), we can say that the failure time of the first (last) activation scheme for the first patient is more (less) than that of the second patient when the shape parameters $\boldsymbol{\alpha}$ and $\boldsymbol{\beta}$ are connected by weakly-supermajorization order. We can similarly use other results for comparing the failure times of the first or last activation schemes.

\subsection{Transportation theory} 
Let us consider some explosives, such as lithium-ion batteries or barrels of highly reactive chemicals, that are transported over long distances. Due to temperature fluctuations or material flaws, some units may be defective and could cause accidents. Let $N$ be the number of defectives, and $X_i$, $i=1,\ldots,N,$ be the distance each defective pack can travel safely. In transportation theory, the random variable $X_{1:N}$ represents the accident-free distance of a shipment of first type explosives. Therefore, comparing the accident-free distances of shipments containing two different types of explosives is an interesting problem from a practical viewpoint. 

Under the above setting, let us consider the random accident-free distances $X_{i}$, for $i=1,\ldots,n,$ to be interdependently distributed as $\text{Kw-E}$ with CDF of the $i$th component being $G(x)=1-[1-\{1-exp^{-(\lambda x)}\}^{\alpha_i}]^{\gamma_i},$ $x,\alpha_i, \gamma_i, \lambda>0.$ Here, we consider the exponential distribution as the parent CDF of the $\text{Kw-G}$ distribution. Suppose two types of explosives $(i)$ Type A and $(ii)$ Type B are being shifted from one place to another. Let $\{X_1,\ldots,X_{N_{1}}\}$ represent the random accident-free distances corresponding to each defective (Type A) and $N_1$ be the number of defective units that could explode, where $X_i\sim\text{Kw- E}(x,\alpha_{i},\gamma_{i}, \mu_{1})$ and $N_1\sim Poission(\lambda_1)$. Also, let $\{Y_1,\ldots,Y_{N_{2}}\}$ represent the random accident-free distances corresponding to each defective (Type B) and $N_2$ be the number of defectives that could explode, where $Y_i\sim\text{Kw- E}(x,\beta_{i},\delta_{i}, \mu_2)$ and $N_2\sim Poission(\mu_2)$. Then, under the given setup, using Theorem \ref{th4}, one can easily state that the accident free distance of Type A shipment is less than that of Type B when the shape parameters $\boldsymbol{\alpha}$ and $\boldsymbol{\beta}$ are connected by weakly-supermajorization order. Similar interpretations can be provided for other theorems as well.}
\section{Concluding remarks}\label{c}
In this article, we have considered two sets of heterogeneous, interdependent random variables {$\{X_{1},\ldots,X_{N_1}\}$ and  $\{Y_{1},\ldots,Y_{N_2}\},$ where the random variables are drawn from $\text{Kw-G}$ family of distributions and also, $N_1$ and $N_2$ are two positive integer-valued random variables independently of $X_{i}'$s and $Y_{i}'$s, respectively.} Several sufficient and necessary conditions have been presented to make stochastic comparisons between two random maxima and minima in the sense of usual stochastic, hazard rate, reversed hazard rate, likelihood ratio and dispersive orders. The concept of Archimedean (survival) copula has been used for dependence of variables. Some counterexamples have been provided to understand the importance of the established sufficient conditions. {Furthermore, we have presented some applications of the established results in biostatistics and transportation theory.} Here, we have only focused on $\text{Kw-G}$ distributed dependent and heterogeneous variables to compare random maxima and minima. So, it will be interesting and challenging to work with different general families of distributions to establish comparison results between random maxima and minima. Moreover, we may also consider two sequence of $\text{Kw-G}$ dependent observations to obtain the comparison results under the same setup. 

{In reliability analysis, risk assessment and decision-making, one may often need to select the most suitable models, systems, or distributions based on their comparative behavior. In this direction, it will be natural to propose a quantitative way to distinguish them by some measures of discrimination. In the context of stochastic ordering, such measures enable us to make a decision about how strongly one system or random variable is dominated by the other. Therefore, developing measures of discrimination based on stochastic ordering results developed in this article will be an interesting research problem to consider. We are currently looking into this problem and hope to report the findings in a future paper. }
\\\\
{\bf Acknowledgement}\\The author Sangita Das gratefully acknowledges the financial support for this research work under NPDF, grant No: PDF/2022/000471, ANRF (SERB), Government of India. The authors also express their sincere thanks to the Editor and the anonymous reviewers for all their useful comments and suggestions on an earlier version of this manuscript which led to this much improved version.\\\\
{\bf Disclosure statement}\\
No potential conflict of interest was reported by the authors.

\end{document}